\documentclass[a4paper,12pt]{article}
\usepackage[utf8]{inputenc}

\usepackage{amssymb}
\usepackage{amsthm}
\usepackage{amsmath}
\usepackage{mathrsfs}
\usepackage{enumerate}
\usepackage{graphicx}
\usepackage{psfrag}

\usepackage[hypertexnames=false]{hyperref} 
\hypersetup{
  colorlinks   = true,    
  urlcolor     = blue,    
  linkcolor    = blue,    
  citecolor    = green      
}

\def\nn{\nonumber}

\newtheorem{theorem}{Theorem}[section]
\newtheorem{definition}[theorem]{Definition}
\newtheorem{proposition}[theorem]{Proposition}
\newtheorem{lemma}[theorem]{Lemma}
\newtheorem{remark}[theorem]{Remark}
\newtheorem{corollary}[theorem]{Corollary}
\newtheorem{convention}[theorem]{Convention}

\def\eps{\varepsilon}
\def\phi{\varphi}
\def\IR{\mathbb{R}}
\def\IN{\mathbb{N}}

\def\IZ{\mathbb{Z}}
\def \diff {\mathrm{d}}

\def\mT{\mathcal{T}}
\def\mA{\mathcal{A}}

\numberwithin{equation}{section}
\title{Generalized Hölder continuity and oscillation functions}
\author{Imre Péter Tóth\\
MTA-BME Stochastics Research Group\\
and\\
Department of Stochastics, Budapest University of Technology and Economics\\
Egry József utca 1, H-507; H-1111 Budapest, Hungary\\
email: mogy@math.bme.hu \quad tel: +36-1-4631101
}

\begin{document}

\maketitle

\begin{abstract}
We study a notion of generalized Hölder continuity for functions on $\IR^d$. We show that for any bounded function $f$ of bounded support and any $r>0$, the $r$-oscillation of $f$ defined as
$osc_r f (x):= \sup_{B_r(x)} f - \inf_{B_r(x)} f$ is automatically generalized Hölder continuous, and we give an estimate for the appropriate (semi)norm. This is motivated by applications in the theory of dynamical systems.

\smallskip
\noindent\emph{Keywords:} Hölder continuity; function oscillation; regularisation; supremum smoothing

\noindent\emph{MSC codes:} 28A75, 37D99
\end{abstract}

\section*{Acknowledgement} This research was supported by Hungarian National Research, Development and Innovation Office
grants No. K 104745 and K 123782, and the Stiftung Aktion Österreich-Ungarn, project OMAA-92öu6. The author is grateful to Péter Bálint, Péter Nándori and Domokos Szász for the illuminating discussions on the problem, and to two anonymous referees for carefully checking the manuscript.

\section{Introduction}

\subsection{Generalized Hölder continuity}\label{sec:genholder-intro}

Let $f:X\to\IR$, where $(X,dist)$ is some metric space. Let $0<\alpha\in\IR$ and $0\le C<\infty$. The function $f$ is said to be Hölder continuous with exponent $\alpha$ and Hölder constant $C$ if for any $x,y\in X$
\begin{equation}\label{eq:Holder-def}
|f(x)-f(y)|\le C dist(x,y)^{\alpha}.
\end{equation}
We now consider $X:=\IR^d$ with the natural Euclidean metric. Following Keller~\cite{Keller85}, Saussol~\cite{S00} and Chernov~\cite{Ch07}, we generalise the above notion so that (\ref{eq:Holder-def}) need not hold for every pair $(x,y)$, only ``on average'' w.r.t Lebesgue measure.
This is motivated by applications in the theory of dynamical systems: in the above quantitative studies of mixing (and also in others), such a generalized Hölder continuity turns out to be the correct notion of regularity, which we need to assume about observables.

In this paper, we will use $B_r(x)$ to denote the open ball of radius $r$ centred at $x\in\IR^d$:
\begin{equation}B_r(x):=\{y\in \IR^d\,:\, |y-x|<r\}.\end{equation}

Let $D\subset \IR^d$ be a Lebesgue measurable set and let $f:D\to\IR$ be any function. For $r>0$ we use $(osc_r f):D\to[0,\infty]$ to denote its ``$r$ oscillation'':
\begin{equation}\label{eq:osc_r-def}
  (osc_r f)(x):=\sup_{y\in B_r(x)\cap D} f(y)-\inf_{y\in B_r(x)\cap D} f(y).
\end{equation}
Of course, $\forall C\in\IR\ \ osc_r f = osc_r (f+C)$.

\begin{lemma}\label{lem:osc-measurable}
 For any Lebesgue measurable $D\subset \IR^d$, any $r>0$ and any $f:D\to\IR$ the oscillation function $osc_r f$ is Lebesgue measurable.
\end{lemma}

\begin{proof}
 In fact, $osc_r f$ is lower semicontinuous: if $(osc_r f)(x)>a$, then there are $y_1,y_2\in B_r(x)\cap D$ such that $f(y_1)-f(y_2)>a$. $B_r(x)$ denotes an \emph{open} sphere, so if $z\in D$ is close enough to $x$, then $y_1,y_2\in B_r(z)\cap D$ as well, so $(osc_r f)(z)>a$.
\end{proof}

\begin{definition}\label{def:GHC}
Let $\mu$ be some constant $c$ times Lebesgue measure on $\IR^d$. For $0<\alpha\le 1$ we define the generalized $\alpha$-Hölder seminorm of $f$ as
\begin{equation}|f|_{\alpha;gH}:=\sup_{r>0}\frac{1}{r^{\alpha}}\int_D (osc_r f)(x)\diff \mu(x) = c \sup_{r>0}\frac{1}{r^{\alpha}}\int_D (osc_r f)(x)\diff x,\end{equation}
where $\diff x$ denotes integration w.r.t. Lebesgue measure.
We say that $f$ is generalized $\alpha$-Hölder continuous if $|f|_{\alpha;gH}<\infty$.
\end{definition}

The factor $c$ is only included for generality -- interesting cases are $c=1$ and $c=\frac{1}{Leb(D)}$.

\begin{remark}\label{rem:sup_not_ess-sup}
This definition coincides with the one given by Chernov in \cite{Ch07}. It is also similar to what Saussol calls the ``quasi-Hölder property'' in \cite{S00} (which is a special case of the notion defined by Keller in \cite{Keller85}). However, it is not exactly the same. The difference is that Keller~\cite{Keller85} and Saussol~\cite{S00} use \emph{essential} supremum and infimum in the definition (\ref{eq:osc_r-def}) of the oscillation, so their definition does not notice the difference between functions that are equal almost everywhere -- w.r.t some distinguished (in our case, Lebesgue) measure.
This is in accordance with using absolutely continuous measures only, when integrating $f$.

From the point of view of the applications we have in mind, two functions, which are equal $\mu$-almost everywhere, may be very different. Indeed, in these applications we integrate $f$ w.r.t. measures which are singular w.r.t. $\mu$ -- actually, concentrated on submanifolds. So, for us, the notion of oscillation with the true $\sup$ and $\inf$ is the good one.
\end{remark}

The main result of this paper is the following theorem.

\begin{theorem}\label{THM:OSC_F-GENHOLDER}
 For any Lebesgue measurable $D\subset\IR^d$, any bounded $f:D\to\IR$, any $r>0$ and any $0<\alpha\le 1$
 \begin{equation}|osc_r f|_{\alpha;gH} \le 2(\sup_D f - \inf_D f)\mu(Conv(D))\left(\frac{2d+1}{r}\right)^\alpha,\end{equation}
 where $Conv(D)$ denotes the convex hull of $D$.
\end{theorem}

The direct motivation for this theorem is the paper \cite{BNSzT17}, where it is explicitly applied in an argument about mixing for a dynamical system. However, the author believes that the result and the proof are of interest on their own.

\begin{remark}\label{rem:what-if-B-closed}
Define a modified version of the oscillation
\begin{equation}\label{eq:closed-osc_r-def}
  (\overline{osc}_r f)(x):=\sup_{y\in \overline{B_r(x)}\cap D} f(y)-\inf_{y\in \overline{B_r(x)}\cap D} f(y)
\end{equation}
using closed balls instead of open ones. Then
\begin{enumerate}
 \item\label{it:closed-osc-measurable} for every $r>0$, $\overline{osc}_r f$ is Lebesgue measurable,
 \item\label{it:closed-osc-thesame} for every $r>0$, $\overline{osc}_r f=osc_r f$ Lebesgue almost everywhere on $D$,
 \item\label{it:closed-seminorm-thesame} $|f|_{\alpha;gH}=\sup_{r>0}\frac{1}{r^{\alpha}}\int_D (\overline{osc}_r f)(x)\diff \mu(x)$,
 \item\label{it:closed-mainthm} Theorem~\ref{THM:OSC_F-GENHOLDER} remains valid for $\overline{osc}$ instead of $osc$.
\end{enumerate}
This will be shown in Remark~\ref{rem:closed-osc_delta_g1_est-h1-h2} and in Section~\ref{sec:what-if-closed}.
\end{remark}

\subsection{Approach map and measure}\label{sec:approach-intro}

In the proof of Theorem~\ref{THM:OSC_F-GENHOLDER}, we need to use a result about ``approach'' maps on $\IR^d$, which take every point the same $\Delta$ distance closer to some target set $H$ (provided they are far enough). The result describes the effect of this approach map on Lebesgue measure.

Let $\emptyset\neq H\subset \IR^d$ and denote its closure by $\bar{H}$. Let $0<\Delta\in\IR$. We define a map $T_\Delta:\IR^d\to\IR^d$ that ``takes points $\Delta$ closer to $H$'' in the following way:
\begin{itemize}
 \item For any $x\in\IR^d$ let $\pi(x)$ be the point in $\bar{H}$ which is closest to $x$ -- that is, the point $\pi(x):=y\in \bar{H}$ where the minimum in $d(x,H)=\min\{d(x,y)\,|\, y\in\bar{H}\}$ is obtained. If there is more than one such $y$, then let $\pi(x)$ be any of them. So $d(x,\pi(x))=d(x,H)$.
 \item Now we define the ``approach map'' $T_\Delta:\IR^d\to\IR^d$ as
 \begin{equation}T_\Delta x:=\begin{cases}
                x+\Delta \frac{\pi(x)-x}{|\pi(x)-x|}, & \text{ if $d(x,H)>\Delta$}\\
                \pi(x), & \text{ if $d(x,H)\le\Delta$}.
               \end{cases}
 \end{equation}
\end{itemize}
\begin{figure}[hbt]
 \psfrag{H}{$H$}
 \psfrag{Tx}{$T_\Delta x$}
 \psfrag{Ty}{$T_\Delta y$}
 \psfrag{x}{$x$}
 \psfrag{y}{$y$}
 \psfrag{d}{$\Delta$}
 \centering
 \includegraphics[scale=0.5]{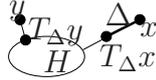}
 \caption{Definition of $T_\Delta$.}
 \label{fig:T_Delta-def}
\end{figure}
See Figure~\ref{fig:T_Delta-def}. This definition implies that
\begin{equation}
d(T_\Delta x,H)=\begin{cases}
                 d(x,H)-\Delta, & \text{ if $d(x,H)>\Delta$}\\
                 0, & \text{ if $d(x,H)\le\Delta$}.
                \end{cases}
\end{equation}

Our main result about $T_\Delta $ is the following.

\begin{theorem}\label{THM:LEBESGUE_SHRINK}
If $\emptyset\neq H\subset \IR^d$, $A\subset\IR^d$ is Lebesgue measurable and $d(H,A)\ge R\ge \Delta\ge 0$, then
\begin{equation}Leb(T_\Delta A)\ge \left(\frac{R-\Delta}{R}\right)^{d-1} Leb(A).\end{equation}
\end{theorem}

We note that $T_\Delta A$ is indeed Lebesgue measurable, as we will see in Remark~\ref{rem:TA_measurable}.
This theorem is quite natural, but the author of this paper could not find it in the literature. It is also optimal: if $H=\{0\}$ and $A=B_{R+\eps}(0)\setminus B_R(0)$, then $T_\Delta A=B_{r-\Delta+\eps}(0)\setminus B_{R-\Delta}(0)$, so
\begin{equation}
 \frac{Leb(T_\Delta A)}{Leb(A)}=\frac{(R-\Delta+\eps)^d-(R-\Delta)^d}{(R+\eps)^d-R^d}\xrightarrow{\eps\searrow 0} \left(\frac{R-\Delta}{R}\right)^{d-1}.
\end{equation}

\subsection{Structure of the paper}

We prove Theorem~\ref{THM:OSC_F-GENHOLDER} in Section~\ref{sec:main-proof}. The proof is self-contained, relying only on elementary measure theory and Theorem~\ref{THM:LEBESGUE_SHRINK}. Theorem~\ref{THM:OSC_F-GENHOLDER} is reduced to more and more elementary (and technical) statements in several steps: intermediate statements are Lemma~\ref{lem:mu1_abscont_mu2}, Lemma~\ref{lem:mTmA_estimate} and claim (\ref{eq:LebTmA_chain_estimate}), which all rely on the next one, with claim (\ref{eq:LebTmA_chain_estimate}) eventually relying on Theorem~\ref{THM:LEBESGUE_SHRINK}. This is done so because giving these statements in reverse order would leave them unmotivated.

The proof of Theorem~\ref{THM:LEBESGUE_SHRINK} is presented in Section~\ref{sec:Lebesgue_shrink_proof}. The proof is self-contained, apart from using 
Theorem 3.2.11 from \cite{Federer69}, called the ``coarea formula''.

Optimality of Theorem~\ref{THM:OSC_F-GENHOLDER} and possible generalizations are briefly discussed in Section~\ref{sec:discussion}.

\section{Proof of Theorem~\ref{THM:OSC_F-GENHOLDER}}
\label{sec:main-proof}

\begin{convention}
 From now on, if we write $\sup_{x\in A} v(x)$ or $\inf_{x\in A} v(x)$ for some $A\subset \IR^d$ and a function $v$ defined on $D$, we mean $\displaystyle\sup_{x\in A\cap D} v(x)$ or $\displaystyle\inf_{x\in A\cap D} v(x)$, respectively.
\end{convention}

\begin{proof}[Proof of Theorem~\ref{THM:OSC_F-GENHOLDER}]
Without loss of generality, we assume that $c=1$. We write
\begin{equation}osc_r f = g_1 - g_2\end{equation}
with
\begin{equation}\label{eq:g1_def}
g_1(x):=\sup_{y\in B_r(x)} f(y),
\end{equation}
\begin{equation}\label{eq:g2_def}
g_2(x):=\inf_{y\in B_r(x)} f(y).
\end{equation}
Clearly
\begin{equation}\label{eq:osc_r_f_estimate-with-g1_g2}
|osc_r f|_{\alpha;gH} \le |g_1|_{\alpha;gH} + |g_2|_{\alpha;gH},
\end{equation}
so it is enough to show that
\begin{equation}\label{eq:g1bound}
|g_1|_{\alpha;gH} \le (\sup_D f - \inf_D f)Leb(Conv(D))\left(\frac{2d+1}{r}\right)^\alpha
\end{equation}
and
\begin{equation}\label{eq:g2bound}
|g_2|_{\alpha;gH} \le (\sup_D f - \inf_D f)Leb(Conv(D))\left(\frac{2d+1}{r}\right)^\alpha.
\end{equation}
We show (\ref{eq:g1bound}). (Then (\ref{eq:g2bound}) is a trivial consequence substituting $f\to(-f)$.)

Let $\hat{D}$ be the closure of $Conv(D)$, and let us extend $f$ to $\hat{D}$ by setting $f(x):=\inf_D f$ when $x\notin D$.
Then $g_1$ remains unchanged on $D$, so $osc_\delta g_1$ can only grow. So the left hand side of (\ref{eq:g1bound}) can only grow,
while the right hand side remains unchanged since $Leb(\hat{D})=Leb(Conv(D))$. So it is enough to show (\ref{eq:g1bound})
for $D$ convex and closed, which we assume from now on.

To show (\ref{eq:g1bound}), we can assume, without loss of generality, that
\begin{equation}\label{eq:f_nonneg_bound}
0\le f \le M:=\sup_D f - \inf_D f.
\end{equation}
Now we take some $\delta>0$, and estimate the integral of $osc_\delta g_1$.

If $\delta\ge \frac{r}{2d+1}$, we use the trivial estimate $osc_\delta g_1\le M$ to get that
\begin{equation}\label{eq:osc-delta-g1-trivial-estimate}
\frac{1}{\delta^\alpha}\int_D osc_\delta g_1 \diff x\le \frac{1}{\delta^\alpha}\int_D M \diff x \le \left(\frac{2d+1}{r}\right)^\alpha M Leb(D),
\end{equation}
which is exactly what we need to show.

So from now on, we assume that
\begin{equation}\label{eq:delta_small}
\delta < \frac{r}{2d+1},
\end{equation}
implying in particular that $\delta<r$. Using the definition (\ref{eq:g1_def}) of $g_1$ we can write
\begin{equation}\label{eq:osc_delta_g1_decomp}
(osc_\delta g_1)(x) =\sup_{y\in B_\delta (x)} \sup_{z\in B_r (y) } f(z) - \inf_{y\in B_\delta (x)} \sup_{z\in B_r (y)} f(z).
\end{equation}
The first term is simply
\begin{equation}\sup_{y\in B_\delta (x)} \sup_{z\in B_r (y)} f(z)\le\sup_{z\in B_{r+\delta} (x)} f(z).\end{equation}
To estimate the second term, notice that for any $y\in B_\delta (x)$, if $|x-z|<r-\delta$, then $|y-z|<r$, so $B_{r-\delta} (x)\subset B_r (y)$ (see Figure~\ref{fig:B_r-d_x_subset_B_r_y}), implying that
\begin{figure}[hbt]
 \psfrag{x}{$x$}
 \psfrag{y}{$y$}
 \psfrag{rmd}{$r-\delta$}
 \psfrag{rpd}{$r+\delta$}
 \psfrag{r}{$r$}
 \psfrag{d}{$\delta$}
 \centering
 \includegraphics[scale=0.5]{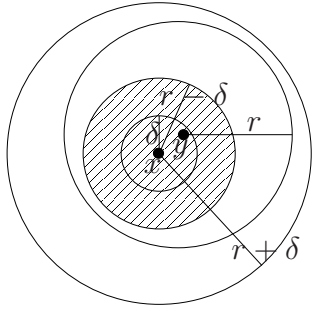}
 \caption{$B_{r-\delta} (x)\subset B_r (y)\subset B_{r+\delta} (x)$ for $y\in B_\delta (x)$.}
 \label{fig:B_r-d_x_subset_B_r_y}
\end{figure}
\begin{equation}\sup_{z\in B_r (y) } f(z)\ge \sup_{z\in B_{r-\delta} (x)} f(z) \quad \text{for any $y\in B_\delta (x)$},\end{equation}
so
\begin{equation}\inf_{y\in B_\delta (x)} \sup_{z\in B_r (y)} f(z) \ge \sup_{z\in B_{r-\delta} (x)} f(z).\end{equation}
Writing these back to (\ref{eq:osc_delta_g1_decomp}) we get that
\begin{equation}\label{eq:osc_delta_g1_est-h1-h2}
osc_\delta g_1 \le h_1 - h_2
\end{equation}
with
\begin{equation}\label{eq:h1-h2-def}
h_1(x):=\sup_{z\in B_{r+\delta} (x)} f(z) \quad , \quad h_2(x):=\sup_{z\in B_{r-\delta} (x)} f(z).
\end{equation}
These $h_1,h_2:D\to\IR$ are easily seen to be Lebesgue measurable (and actually lower semicontinuous), just like $osc_r f$ in Lemma~\ref{lem:osc-measurable}.

\begin{remark}\label{rem:closed-osc_delta_g1_est-h1-h2}
If we discuss $\overline{osc}_r f$ defined in (\ref{eq:closed-osc_r-def}) instead of $osc_r f$, we can define $\bar{g}_1$ and $\bar{g}_2$ using closed balls instead of open ones in (\ref{eq:g1_def}) and (\ref{eq:g2_def}). Then (\ref{eq:osc-delta-g1-trivial-estimate}) remains true for $\bar{g}_1$. (\ref{eq:osc_delta_g1_decomp}) becomes
\begin{equation}
(osc_\delta \bar{g}_1)(x) =\sup_{y\in B_\delta (x)} \sup_{z\in \overline{B_r (y)} } f(z) - \inf_{y\in B_\delta (x)} \sup_{z\in \overline{B_r (y)}} f(z).
\end{equation}
This implies
\begin{equation}
osc_\delta \bar{g}_1 \le h_1 - h_2
\end{equation}
with the \emph{same} $h_1$ and $h_2$ as in (\ref{eq:osc_delta_g1_est-h1-h2}), so the rest of the proof remains unchanged. This proves item~\ref{it:closed-mainthm} of Remark~\ref{rem:what-if-B-closed}.
\end{remark}

We want to estimate $\int_D osc_\delta g_1 \le \int_D h_1 - \int_D h_2$ from above. The idea is roughly that if some $u\in[0,M]$ is obtained as $u=h_1(x)$
for some $x\in D$, then the same $u$ is also obtained as $u=h_2(\tilde x)$ for some (possibly other) $\tilde x\in D$.
Moreover, the set of such $\tilde x$ cannot be much smaller (in terms of Lebesgue measure), than the set of the $x$.

To formalise the argument, let $\mu_1$ and $\mu_2$ be measures on $\IR$, which are the push-forwards of Lebesgue measure from $D$ to $\IR$ by $h_1$ and $h_2$, respectively:
for any Borel set $A\subset \IR$
\begin{equation}\mu_1(A):=Leb(h_1^{-1}(A)), \quad\quad \mu_2(A):=Leb(h_2^{-1}(A)).\end{equation}

Notice that both $\mu_1$ and $\mu_2$ are concentrated on $[0,M]$. So integral substitution gives
\begin{equation}\label{eq:int_D_h-substitution}
\int_D h_1(x)\diff x=\int_{[0,M]} u\diff \mu_1(u) \quad , \quad \int_D h_2(x)\diff x=\int_{[0,M]} u\diff \mu_2(u).
\end{equation}

The idea above is made precise in the following lemma:

\begin{lemma}\label{lem:mu1_abscont_mu2}
 If $\delta < \frac{r}{2d+1}$, then $\mu_1$ is absolutely continuous w.r.t. $\mu_2$, with density
 \begin{equation}\frac{\diff\mu_1}{\diff\mu_2}\le C = C(r,\delta,d):=\frac{1}{1-d\frac{2\delta}{r-\delta}}.\end{equation}
\end{lemma}
We postpone the proof of this lemma, and finish the proof of the theorem using the lemma.

The lemma implies
\begin{equation}\int_{[0,M]} u\diff\mu_2(u)\ge \int_{[0,M]} u \frac{1}{C} \diff\mu_1(u) = \frac{1}{C} \int_{[0,M]} u \diff\mu_1(u),\end{equation}
so
\begin{eqnarray}
\int_{[0,M]} u\diff\mu_1(u)-\int_{[0,M]} u \diff\mu_2(u)\le \left(1-\frac{1}{C}\right)\int_{[0,M]} u \diff\mu_1(u)\le \nn \\
\le \left(1-\frac{1}{C}\right) M \mu_1([0,M])= \left(1-\frac{1}{C}\right) M Leb(D).
\end{eqnarray}
The constant factor is $1-\frac{1}{C(r,\delta,d)}=d\frac{2\delta}{r-\delta}$. Our assumption (\ref{eq:delta_small}) implies that
$r-\delta>\frac{2dr}{2d+1}$, so $\frac{1}{r-\delta}<\frac{2d+1}{2d}\frac{1}{r}$. So $1-\frac{1}{C(r,\delta,d)}\le (2d+1)\frac{\delta}{r}$.
Writing this back to (\ref{eq:osc_delta_g1_est-h1-h2}) using (\ref{eq:int_D_h-substitution}) gives
\begin{equation}\label{eq:h1-minus-h2-small}
\int_D (osc_\delta g_1)(x)\diff x \le \int_D h_1 \diff x - \int_D h_2 \diff x \le (2d+1)\frac{\delta}{r} M Leb(D).
\end{equation}

\begin{remark}\label{rem:continuity-in-r}
Looking at the definition (\ref{eq:h1-h2-def}) of $h_1$ and $h_2$, the second inequality in (\ref{eq:h1-minus-h2-small}) implies that if $Leb(Conv(D))<\infty$ and $f$ is bounded, then the function
$r\mapsto \int_D \sup_{z\in B_r(x)} f(z) \diff x$, which is clearly monotone increasing, is actually continuous at every $r>0$.
\end{remark}

Using again the assumption (\ref{eq:delta_small}) we get
\begin{eqnarray}
\frac{1}{\delta^\alpha}\int_D (osc_\delta g_1)(x)\diff x\le \delta^{1-\alpha} \frac{2d+1}{r} M Leb(D)\le \nn \\
\left(\frac{r}{2d+1}\right)^{1-\alpha}\frac{2d+1}{r} M Leb(D)= \left(\frac{2d+1}{r}\right)^\alpha M Leb(D),
\end{eqnarray}
which is again exactly what we need to show. Theorem~\ref{THM:OSC_F-GENHOLDER} is proven.
\end{proof}

We are left to prove Lemma~\ref{lem:mu1_abscont_mu2}. We will use the notation $H^{(r)}$ to denote the open $r$-neighbourhood of $H\subset \IR^d$ within $D$:
\begin{equation}H^{(r)}:=\{z\in D\,:\, dist(z,H)<r\}.\end{equation}

\begin{proof}[Proof of Lemma~\ref{lem:mu1_abscont_mu2}.]
For any open interval $I=(a,b)\subset\IR$ we need to show that $\mu_1(I)\le C \mu_2(I)$, which is the same as
\begin{equation}\label{eq:h1_h2_comparison}
Leb(h_1^{-1}(I))\le C Leb(h_2^{-1}(I)).
\end{equation}
To avoid a trivial case, we assume that $h_1^{-1}(I)$ is non-empty. Let
\begin{equation}H:=f^{-1}(I)\subset D.\end{equation}
Now if $x\in h_1^{-1}(I)$, meaning that $\sup_{z\in B_{r+\delta} (x)} f(x)\in I$, then $\exists z\in B_{r+\delta}(x) \cap H$, so $dist(x,H)<r+\delta.$
(This also means that since $h_1^{-1}(I)$ is non-empty, $H$ is also non-empty.)
Using such an $x$, we construct two candidate points, one of which is certainly in $h_2^{-1}(I)$. See Figure~\ref{fig:candidate_points_x_and_Tx}.
\begin{figure}[hbt]
 \psfrag{H_r+d}{$H^{(r+\delta)}$}
 \psfrag{H_r-d}{$H^{(r-\delta)}$}
 \psfrag{r+d}{$r+\delta$}
 \psfrag{r-d}{$r-\delta$}
 \psfrag{2d}{$2\delta$}
 \psfrag{H}{$H$}
 \psfrag{Tx2}{$Tx_2$}
 \psfrag{x2}{$x_2$}
 \psfrag{x1}{$x_1$}
 \psfrag{z1}{$z_1$}
 \centering
 \includegraphics[scale=0.5]{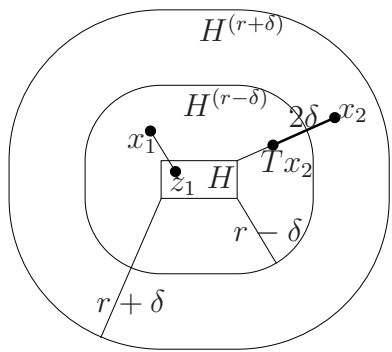}
 \caption{Candidate points for being in $h_2^{-1}(I)$: from $x_1$ we get $x_1$ itself; from $x_2$ we get $Tx_2$.}
 \label{fig:candidate_points_x_and_Tx}
\end{figure}
\begin{enumerate}[a.)]
 \item The first candidate point is $x$ itself. If $dist(x,H)<r-\delta$ happens to hold, then $\exists z\in B_{r-\delta} (x)$ such that $f(z)\in I$, so $h_2(x)\ge f(z)>a$.
 On the other hand, $h_2(x)\le h_1(x)<b$, so $h_2(x)\in I$ and so $x\in h_2^{-1}(I)$.
 \item To construct the other candidate point, we define a map $T$ on $\IR^d\setminus H^{(2\delta)}$ that ``takes points $2\delta$ closer to $H$''. To be precise, for any $x\in \IR^d$ with $dist(x,H)\ge 2\delta$, let
 $\pi(x)$ be the point in $\bar H$ which is nearest to $x$.\footnote{If there is more than one such point, let $\pi(x)$ be any one of them. This causes no problem,
 because there is only one such point for almost every $x$. The details are written in the proof of Theorem~\ref{THM:LEBESGUE_SHRINK} and in Remark~\ref{rem:TA_measurable}.}
 Now define
 \begin{equation}Tx:=x+2\delta\frac{\pi(x)-x}{|\pi(x)-x|}.\end{equation}
 Since $D$ was assumed to be closed and convex, if $x\in D$ then $\pi(x)\in D$ and $Tx\in D$.
 This $Tx$ also satisfies $dist(Tx,H)=dist(x,H)-2\delta\le r+\delta-2\delta=r-\delta$, so again $h_2(Tx)>a$.
 On the other hand, $B_{r-\delta} (Tx)\subset B_{r+\delta} (x)$, so $h_2(Tx)\le h_1(x)<b$. We got
 $h_2(Tx)\in I$, so $Tx\in h_2^{-1}(I)$.
\end{enumerate}
Notice that since $\delta < \frac{r}{2d+1}\le \frac{r}{3}$ by assumption, either $dist(x,H)<r-\delta$ or $dist(x,H)\ge 2\delta$ certainly holds, so for any $x\in D$ either $x\in h_2^{-1}(I)$ or $Tx$ is well defined and $Tx\in h_2^{-1}(I)\subset D$. To write this concisely, we introduce the operation $\mT$ on subsets of $\IR^d$ as
\begin{equation}\mT\mA :=(\mA\cap H^{(r-\delta)})\cup T\mA,\end{equation}
where $T\mA$ is meant by just ignoring points of $\mA$ where $T$ is undefined.
With this notation, we just saw that
\begin{equation}\mT(h_1^{-1}(I))\subset h_2^{-1}(I),\end{equation}
so $Leb(h_2^{-1}(I))$ can be estimated from below as
\begin{equation}Leb(h_2^{-1}(I)) \ge Leb\left(\mT(h_1^{-1}(I))\right).\end{equation}

Now (\ref{eq:h1_h2_comparison}) and thus Lemma~\ref{lem:mu1_abscont_mu2} is an immediate consequence of the following Lemma~\ref{lem:mTmA_estimate}.
\end{proof}

\begin{lemma}\label{lem:mTmA_estimate}
For any Lebesgue measurable $\mA\subset H^{(r+\delta)}$
\begin{equation}Leb(\mT\mA)\ge \left(1-d\frac{2\delta}{r-\delta}\right) Leb(\mA).\end{equation}
\end{lemma}

\begin{proof}
If $\mA\subset H^{(r-\delta)}$, then $\mA\subset\mT\mA$, so the statement is trivial. When this is not the case, we will need to understand
the effect of $\mT$ very precisely. For this purpose, we cut up $\mA\setminus H^{(r-\delta)}$ into disjoint sets $A_k$, based on the number of iterations of $T$ that we can perform without leaving $\mA$. The points that can be reached with such iterations will be treated with careful calculations. For the rest, the trivial estimate suffices.

The proof is based on the properties of the map $T$ studied in Section~\ref{sec:Lebesgue_shrink_proof}. Strictly speaking we will only use Theorem~\ref{THM:LEBESGUE_SHRINK} about the limited effect of $T$ on Lebesgue measure. The essence of the understanding is that as long as $dist(A,H)>2\delta$, the map $T$ is one-to-one on $A$ and $TA$ is not much smaller than $A$.

First, let
\begin{equation}K:=\left\lfloor\frac{r}{2\delta}-\frac12\right\rfloor=\max\{k\in\IN\,:\,r-(2k+1)\delta\ge 0\}.\end{equation}
With this definition, for any point $x\in H^{(r+\delta)}\setminus H^{(r-\delta)}$, $T^k x$ makes sense for $k=0,1,\dots,K$, and possibly for $k=K+1$, but certainly not for $k=K+2$, because
$0\le dist(T^K x,H)<4\delta$. For a set $A\subset H^{(r+\delta)}\setminus H^{(r-\delta)}$, the first $K(+1)$ iterates $A,TA,T^2A,\dots,T^K A$ are disjoint, and of comparable measure.
The next iterate $T^{K+1} A$, even if non-empty, can have arbitrarily small measure, so we don't care if it is empty or not, and we will not make use of it in our estimates. This justifies the following definitions -- see also Figure~\ref{fig:mTmA_estimate_notation}:
\begin{figure}[hbt]
 \psfrag{CAK-1}{$\mA_{K-1}$}
 \psfrag{CAK}{$\mA_K$}
 \psfrag{CA2}{$\mA_2$}
 \psfrag{CA1}{$\mA_1$}
 \psfrag{CA}{$\mA$}
 \psfrag{TK+1AK}{$T^{K+1}A_K$}
 \psfrag{TK-1AK-1}{$T^{K-1}A_{K-1}$}
 \psfrag{TK-1AK}{$T^{K-1}A_K$}
 \psfrag{TKAK}{$T^KA_K$}
 \psfrag{T2AK-1}{$T^2A_{K-1}$}
 \psfrag{T2AK}{$T^2A_K$}
 \psfrag{TAK}{$TA_K$}
 \psfrag{TAK-1}{$TA_{K-1}$}
 \psfrag{AK-1}{$A_{K-1}$}
 \psfrag{AK}{$A_K$}
 \psfrag{T2A2}{$T^2A_2$}
 \psfrag{TA2}{$TA_2$}
 \psfrag{A2}{$A_2$}
 \psfrag{TA1}{$TA_1$}
 \psfrag{A1}{$A_1$}
 \psfrag{A0}{$\mA_0=A_0$}
 \psfrag{r-(2K+1)d}{$r-(2K+1)\delta$}
 \psfrag{r-3d}{$r-3\delta$}
 \psfrag{r-d}{$r-\delta$}
 \psfrag{r+d}{$r+\delta$}
 \psfrag{A*}{$\mA^*$}
 \psfrag{H}{$H$}
 \centering
 \includegraphics[bb=0 225 612 606,clip,scale=0.5]{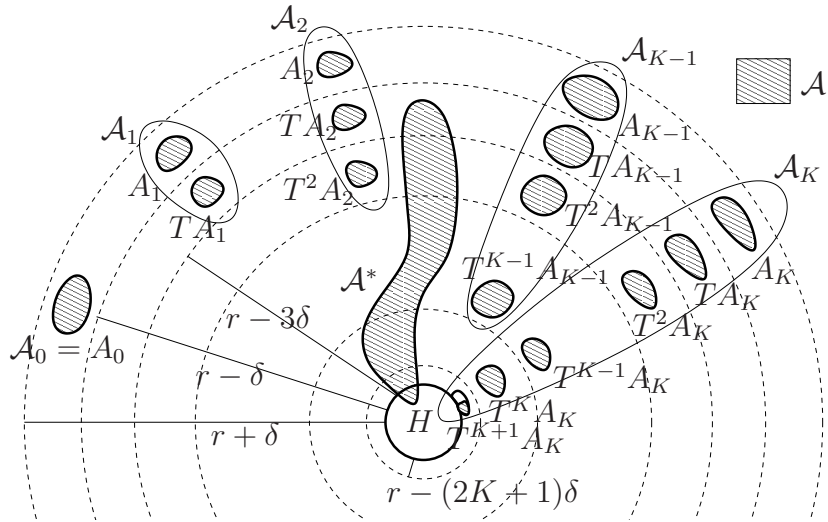}
 \caption{Notation for the proof of Lemma~\ref{lem:mTmA_estimate}.}
 \label{fig:mTmA_estimate_notation}
\end{figure}
\noindent For $k=0,1,\dots,K-1$
\begin{eqnarray}
A_k&:=&\{x\in \mA\setminus H^{(r-\delta)}\,:Tx\in\mA,T^2 x\in\mA,\dots,T^k x\in\mA,\text{ but }T^{k+1} x\notin\mA\} \nn \\
\mA_k&:=&A_k\cup TA_k \cup \dots \cup T^k A_k.\label{eq:mA_k_def}
\end{eqnarray}
On the other hand, for $k=K$,
\begin{eqnarray}
A_K&:=&\{x\in \mA\setminus H^{(r-\delta)}\,:Tx\in\mA,T^2 x\in\mA,\dots,T^K x\in\mA\} \nn \\
\mA_K&:=&A_K\cup TA_K \cup \dots \cup T^K A_K \cup (T^{K+1}A_K\cap \mA).\label{eq:mA_K_def}
\end{eqnarray}
For the rest,
\begin{equation}\mA^*:=\mA\setminus \bigcup_{k=0}^K \mA_k.\end{equation}
These definitions make sure that
\begin{equation}\mA = \mA_0 \cup \mA_1 \cup \dots \cup \mA_K \cup \mA^*\end{equation}
is a disjoint union, and more importantly, the union
\begin{equation}\mT\mA = \mT\mA_0 \cup \mT\mA_1 \cup \dots \cup \mT\mA_K \cup \mT\mA^*\end{equation}
is also disjoint. This makes the estimation of $Leb(\mT\mA)$ from below feasible.
In fact, $\mT\mA_k=T\mA_k$ for every $k$, while $\mT\mA^*\supseteq\mA^*$.

The lemma follows from the following claim: for every $k=0,1,\dots,K$
\begin{equation}\label{eq:LebTmA_chain_estimate}
Leb(T\mA_k)\ge \left(1-d\frac{2\delta}{r-\delta}\right) Leb(\mA_k).
\end{equation}

Indeed, using the claim, with the notation $\frac1C=\left(1-d\frac{2\delta}{r-\delta}\right)<1$,
\begin{eqnarray}
Leb(\mT\mA)&=&\sum_{k=0}^K Leb(\mT\mA_k)+Leb(\mT\mA^*)\ge\nn\\
&\ge&\sum_{k=0}^K Leb(T\mA_k)+Leb(\mA^*)\ge\\
&\ge& \sum_{k=0}^K \frac1C Leb(\mA_k)+\frac1C Leb(\mA^*)=\nn\\
&=& \frac1C Leb(\mA),\nn
\end{eqnarray}
which is exactly what we have to prove. So we are left to show the claim (\ref{eq:LebTmA_chain_estimate}).

The key to the calculation is Theorem~\ref{THM:LEBESGUE_SHRINK}, which says in our case that
if $2\delta\le\rho\in\IR$ and $X\subset D$ is Lebesgue measurable such that $dist(X,H)\ge\rho$, then
\begin{equation}Leb(TX)\ge \left(\frac{\rho-2\delta}{\rho}\right)^{d-1} Leb(X).\end{equation}
We use this with $X=T^j A_k$ and $\rho:=r-(2j+1)\delta\le d(T^j A_k,H)$ (for $0\le j<k\le K$), to get
\begin{equation}\frac{Leb(T^{j+1}A_k)}{Leb(T^j A_k)}\ge  \left(\frac{r-(2j+3)\delta}{r-(2j+1)\delta}\right)^{d-1}\end{equation}
for all $j<k$, which implies by induction that
\begin{equation}\frac{Leb(T^j A_k)}{Leb(A_k)}\ge  \left(\frac{r-(2j+1)\delta}{r-\delta}\right)^{d-1}\end{equation}
for all $j\le k$. The sets $A_k,TA_k,T^2 A_k,\dots, T^k A_k$ are pairwise disjoint, so (\ref{eq:mA_k_def}) and (\ref{eq:mA_K_def}) give
\begin{equation}\label{eq:Leb_ma_k_estimate}
Leb(\mA_k)\ge \sum_{j=0}^k \left(\frac{r-(2j+1)\delta}{r-\delta}\right)^{d-1} Leb(A_k).
\end{equation}

Our next goal is to estimate $\frac{Leb(\mA_k)-Leb(T\mA_k)}{Leb(\mA_k)}$ from above by estimating the numerator from above and the denominator from below.
We make a fine distinction between the cases $k<K$ and $k=K$.
\begin{enumerate}[a.)]
 \item If $k<K$, meaning that $dist(T^k A_k,H)\ge 2\delta$, then ``there is room for a $T^{k+1}A_k$'', so
 \begin{equation}T(\mA_k)=TA_k\cup T^2 A_k \cup \dots \cup T^{k+1} A_k,\end{equation}
 and $Leb(T^{k+1}A_k)\ge \left(\frac{r-(2k+3)\delta}{r-\delta}\right)^{d-1} Leb(A_k)$.
 Now
 \begin{eqnarray}
 Leb(\mA_k)-Leb(T\mA_k)=Leb(A_k)-Leb(T^{k+1}A_k)\le \\
 \le \left[1-\left(\frac{r-(2k+3)\delta}{r-\delta}\right)^{d-1}\right]Leb(A_k).
 \end{eqnarray}
 We estimate the sum in (\ref{eq:Leb_ma_k_estimate}) with an integral: since the function $t\mapsto\left(\frac{r-(2t+1)\delta}{r-\delta}\right)^{d-1}$ is monotone decreasing on $[0,k+1]$, the sum in (\ref{eq:Leb_ma_k_estimate}) is an upper integral-approximating sum, so
 \begin{eqnarray}\frac{Leb(\mA_k)}{Leb(A_k)}&\ge& 
 \int_0^{k+1} \left(\frac{r-(2t+1)\delta}{r-\delta}\right)^{d-1}\diff t \nn \\
 &=&\frac{1}{d}\frac{r-\delta}{2\delta}\left[1-\left(\frac{r-(2k+3)\delta}{r-\delta}\right)^d\right].
 \end{eqnarray}
 Putting these together, and using that $0\le\frac{r-(2k+3)\delta}{r-\delta}<1$, we get that
 \begin{equation}
 \frac{Leb(\mA_k)-Leb(T\mA_k)}{Leb(\mA_k)}\le
 d\frac{2\delta}{r-\delta} \frac{1-\left(\frac{r-(2k+3)\delta}{r-\delta}\right)^{d-1}}{1-\left(\frac{r-(2k+3)\delta}{r-\delta}\right)^d}\le d\frac{2\delta}{r-\delta}.
 \end{equation}
 \item If $k=K$, then we use
 \begin{equation}Leb(\mA_K)-Leb(T(\mA_K))\le Leb(A_K)\end{equation}
 and again an integral to estimate the sum in (\ref{eq:Leb_ma_k_estimate}):
 \begin{equation}\frac{Leb(\mA_K)}{Leb(A_K)}\ge \int_0^{\frac{r}{2\delta}-\frac12} \left(\frac{r-(2t+1)\delta}{r-\delta}\right)^{d-1}\diff t=\frac{1}{d}\frac{r-\delta}{2\delta}.\end{equation}
 (Note the careful choice of the upper integration boundary: the function $t\mapsto\left(\frac{r-(2t+1)\delta}{r-\delta}\right)^{d-1}$ is nonnegative and monotone decreasing on $[0,\frac{r}{2\delta}-\frac12]$, and $K=\left\lfloor\frac{r}{2\delta}-\frac12\right\rfloor$.)
 Putting these together, we get that
 \begin{equation}
 \frac{Leb(\mA_K)-Leb(T\mA_K)}{Leb(\mA_K)}\le d\frac{2\delta}{r-\delta},
 \end{equation}
 just like in the previous case.
\end{enumerate}
It immediately follows that
\begin{equation}\frac{Leb(T\mA_k)}{Leb(\mA_k)}\ge 1-d\frac{2\delta}{r-\delta},\end{equation}
which is exactly the claim (\ref{eq:LebTmA_chain_estimate}).
\end{proof}

\section{Proof of Theorem~\ref{THM:LEBESGUE_SHRINK}}\label{sec:Lebesgue_shrink_proof}

We prove Theorem~\ref{THM:LEBESGUE_SHRINK} through a few lemmas and propositions. The first statement is about the ``infinitesimal'' version of the approach map $T_\Delta$, when $\Delta$ is very small.
We claim that if two points are far away from $H$, then such a $T_\Delta$ does not bring them much closer to each other:

\begin{lemma}\label{lem:T_t_weak_contraction}
 Let $\tilde{x},\tilde{y}\in\IR^d$ with $d(\tilde{x},H)\ge r$ and $d(\tilde{y},H)\ge r$. Let $\tilde{f}(s)=d(T_s \tilde{x}, T_s \tilde{y})$. Then the derivative of $\tilde{f}$ at $0$ can be negative, but not too much:
 \begin{equation}-\dot{\tilde{f}}(0)\le \frac{\tilde{f}(0)}{r}.\end{equation}
\end{lemma}
\begin{figure}[hbt]
 \psfrag{H}{$H$}
 \psfrag{pix}{$\pi(\tilde{x})$}
 \psfrag{piy}{$\pi(\tilde{y})$}
 \psfrag{x}{$\tilde{x}$}
 \psfrag{y}{$\tilde{y}$}
 \psfrag{Y}{$Y$}
 \psfrag{R1}{$R_1$}
 \psfrag{R2}{$R_2$}
 \psfrag{a}{$a$}
 \psfrag{b}{$b$}
 \centering
 \includegraphics[scale=0.75]{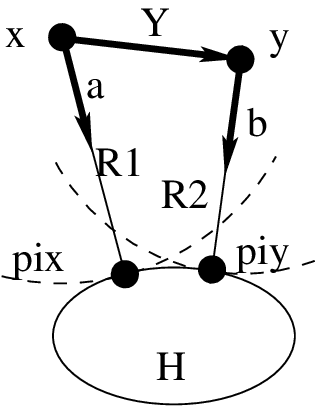}
 \caption{Notation in the proof of Lemma~\ref{lem:T_t_weak_contraction}.}
 \label{fig:T_t_weak_contraction_notation}
\end{figure}

\begin{proof}
For the notation, see Figure~\ref{fig:T_t_weak_contraction_notation}. Let $Y=\tilde{y}-\tilde{x}$, $a=\frac{\pi(\tilde{x})-\tilde{x}}{|\pi(\tilde{x})-\tilde{x}|}$,
$b=\frac{\pi(\tilde{y})-\tilde{y}}{|\pi(\tilde{y})-\tilde{y}|}$, $R_1=|\pi(\tilde{x})-\tilde{x}|$, $R_2=|\pi(\tilde{y})-\tilde{y}|$. So $a^2=b^2=1$, $\pi(\tilde{x})=\tilde{x}+R_1 a$ and $\pi(\tilde{y})=\tilde{y}+R_2 b=\tilde{x}+Y+R_2 b$. We use the fact that $\pi(\tilde{x})$ is the
nearest point of $H$ to $\tilde{x}$, so in particular $d(\pi(\tilde{y}),\tilde{x})\ge d(\pi(\tilde{x}),\tilde{x})$. Similarly, $\pi(\tilde{y})$ is the nearest point of $H$ to $\tilde{y}$, so $d(\pi(\tilde{x}),\tilde{y})\ge d(\pi(\tilde{y}),\tilde{y})$. With the above notation these can be written as
$|Y + R_2 b|\ge R_1$ and $|R_1 a -Y|\ge R_2$, which are equivalent to
\begin{equation}\label{eq:bY_estimate}
bY \ge \frac{R_1^2-R_2^2-Y^2}{2R_2},
\end{equation}
\begin{equation}\label{eq:aY_estimate}
aY \le \frac{R_1^2-R_2^2+Y^2}{2R_1}.
\end{equation}
An explicit calculation gives $\tilde{f}(t)=|\tilde{y}+tb-(\tilde{x}+ta)|=|Y+t(b-a)|$, so $\tilde{f}(0)=|Y|$ and
\begin{equation}-\dot{\tilde{f}}(0)=\frac{1}{|Y|}Y(a-b)=\frac{Ya-Yb}{|Y|}.\end{equation}
This can be estimated from above directly using the assumptions as formulated in (\ref{eq:bY_estimate}) and (\ref{eq:aY_estimate}) to give
\begin{eqnarray}
-\dot{\tilde{f}}(0) &\le& \frac{1}{|Y|} \left[\frac{R_1^2-R_2^2+Y^2}{2R_1}- \frac{R_1^2-R_2^2-Y^2}{2R_2} \right]=\nn\\
            & = &\frac{1}{|Y|}\frac12\left(\frac{1}{R_1}+\frac{1}{R_2}\right)\left(Y^2-(R_1-R_2)^2\right).
 \end{eqnarray}
Using $\frac12\left(\frac{1}{R_1}+\frac{1}{R_2}\right)\le\frac{1}{r}$ and $(R_1-R_2)^2\ge 0$ we get
\begin{equation}-\dot{\tilde{f}}(0)\le \frac{1}{|Y|}\frac{1}{r}Y^2=\frac{\tilde{f}(0)}{r}.\end{equation}
\end{proof}

\begin{corollary}\label{cor:T_t_weak_contraction2}
 Let $x,y\in\IR^d$, $d(x,H)\ge R$ and $d(y,H)\ge R$. Let $f(t)=d(T_t x, T_t y)$. Then for every $0\le t\le R$
 \begin{equation}-\dot{f}(t)\le \frac{f(t)}{R-t}.\end{equation}
\end{corollary}

\begin{proof}
 Fix some $0\le t\le R$. Let $\tilde{x}= T_t x$, $\tilde{y}= T_t y$ and $r=R-t$. Then $\pi(\tilde{x})=\pi(x)$, $\pi(\tilde{y})=\pi(y)$ and the conditions
 of Lemma~\ref{lem:T_t_weak_contraction} are satisfied. Moreover, $\tilde{f}(s)=f(t+s)$, so $f(t)=\tilde{f}(0)$ and $\dot{f}(t)=\dot{\tilde{f}}(0)$.
 Applying the lemma gives exactly the statement of the corollary.
\end{proof}

\begin{proposition}\label{prop:T_Delta_weak_contraction}
 If $x,y\in\IR^d$, $d(x,H)\ge R$ and $d(y,H)\ge R$, then for any $0\le\Delta\le R$
 \begin{equation}d(T_\Delta x, T_\Delta y)\ge \frac{R-\Delta}{R}d(x,y).\end{equation}
\end{proposition}

\begin{proof}
To avoid a trivial case, assume $d(x,y)\neq 0$. We apply Corollary~\ref{cor:T_t_weak_contraction2}. With the function $f$ introduced there,
$d(x,y)=f(0)$, $d(T_\Delta x, T_\Delta y)=f(\Delta)$, and the statement of the corollary can be read as
\begin{equation}\frac{\diff}{\diff t}(\ln f(t))\ge -\frac{1}{R-t}.\end{equation}
This implies that
\begin{equation}\ln\frac{f(\Delta)}{f(0)}=\ln f(\Delta) - \ln f(0) \ge \int_0^\Delta \frac{-1}{R-t}\diff t = \ln\frac{R-\Delta}{R}.\end{equation}
So
\begin{equation}\frac{d(T_\Delta x, T_\Delta y)}{d(x,y)}=\frac{f(\Delta)}{f(0)}\ge \frac{R-\Delta}{R}.\end{equation}
\end{proof}

We are interested in the effect of such an approach map on the measure of sets. So for $B\subset\IR^d$ and $0\le s\le d$ let $\mathcal{H}^s(B)$ denote the s-dimensional outer Hausdorff measure of $B$.
The next statement is an easy corollary of the previous.

\begin{proposition}\label{prop:Hausdorff_shrink}
If $\emptyset\neq H,A\subset\IR^d$, $d(H,A)\ge R\ge \Delta\ge 0$ and $0\le s\le d$, then
\begin{equation}\mathcal{H}^s(T_\Delta A)\ge \left(\frac{R-\Delta}{R}\right)^s \mathcal{H}^s(A)\end{equation}
(with the convention $0^0:=0$ for the case $s=0$, $R=\Delta$).
\end{proposition}

\begin{proof}
 If $\Delta=R$, the statement is trivial. If $\Delta<R$, then the first implication of Proposition~\ref{prop:T_Delta_weak_contraction} is that $T_\Delta$ is injective, so
 \begin{equation}A=\{T_\Delta^{-1}y\,|\,y\in T_\Delta A\}.\end{equation}
 As a result, if $\{U_k\}_{k=1}^\infty$ is a covering of $T_\Delta A$, then we can cover $A$ with $\{U_k^-\}_{k=1}^\infty$, where
 $U_k^- := T_\Delta^{-1}(U_k\cap T_\Delta A)$. Proposition~\ref{prop:T_Delta_weak_contraction} implies that
 \begin{equation}\label{eq:diam_Uk}
 diam(U_k)\ge \frac{R-\Delta}{R} diam(U_k^-).
 \end{equation}
 But by definition, the outer Hausdorff measure is essentially an infimum of $\sum_k diam(U_k)^s$ over coverings $\{U_k\}$:
 \begin{equation}\label{eq:HM_def}
  \mathcal{H}^s(A)= \lim_{\delta\searrow 0} \mathcal{H}^s_\delta (A)
 \end{equation}
 where
 \begin{equation}
 \mathcal{H}^s_\delta(A)=c_s \inf\left\{\sum_{k=1}^\infty diam(V_k)^s \,|\,diam(V_k)\le\delta, A\subset \bigcup_{k=1}^\infty V_k\right\}
 \end{equation}
 and $c_s$ is some normalising constant.
 So (\ref{eq:diam_Uk}) implies that
 \begin{equation}
 \mathcal{H}^s_\delta(T_\Delta A) \ge \left(\frac{R-\Delta}{R}\right)^s \mathcal{H}^s_{\frac{R}{R-\Delta}\delta}(A).
 \end{equation}
 So the definition (\ref{eq:HM_def}) gives the statement of the proposition.
 \end{proof}

Applying this proposition with $s=d$ would immediately give a comparison of Lebesgue measures. Our goal, Theorem~\ref{THM:LEBESGUE_SHRINK}, is only a little stronger. We will get it by utilising the fact that Proposition~\ref{prop:T_Delta_weak_contraction} is a worst case estimate for the contraction, and there is a direction in which $T_\Delta$ does not contract at all.

\begin{proof}[Proof of Theorem~\ref{THM:LEBESGUE_SHRINK}]
We will apply the theory of area and coarea of Lipschitz continuous maps from \cite{Federer69}, section 3.2.

Let $f:\IR^d\to\IR^+$ be defined as $f(x):=d(x,H)$. This $f$ is clearly Lipschitz continuous with Lipschitz constant $1$, so it is Lebesgue almost everywhere differentiable.
Consider an $x\notin \bar{H}$, so $f(x)>0$. If ``the point $\pi(x)$ in $\bar{H}$ nearest to $x$'' is not well defined, because there are $y_1\neq y_2\in\bar{H}$ such that $d(x,y_1)=d(x,y_2)=d(x,H)$, then the (one-sided) directional derivative of $f$ at $x$ is $-1$ in both the direction of $y_1$ and $y_2$, so $f$ can not be differentiable at $x$. As a result, this can only happen for a zero Lebesgue measure set of $x$. On the remaining full measure set of $x\notin \bar{H}$, $\pi(x)$ is well defined, the directional derivative of $f$ is $-1$ and thus the gradient is the unit vector $\nabla f(x)=\frac{x-\pi(x)}{|x-\pi(x)|}$. In the language of \cite{Federer69}, section 3.2, this means that the 1-dimensional Jacobian is $J_1 f=1$ almost everywhere outside $\bar{H}$.

We foliate $A$ and $T_\Delta A$ with level sets of this function $f$ -- see Figure~\ref{fig:Lebesgue_shrink_proof_foliation}.
\begin{figure}[hbt]
 \psfrag{TA}{$T_\Delta A$}
 \psfrag{A}{$A$}
 \psfrag{F1}{$\{d(.,H)=t+\Delta\}$}
 \psfrag{F2}{$\{d(.,H)=t\}$}
 \centering
 \includegraphics[scale=0.75]{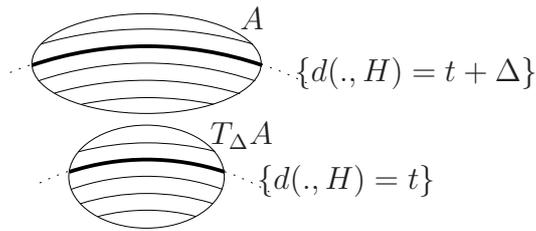}
 \caption{Foliation of $A$ and $T_\Delta A$ with level sets of $f$.}
 \label{fig:Lebesgue_shrink_proof_foliation}
\end{figure}
The $d$-dimensional Lebesgue measure of $A$ and $T_\Delta A$ can be calculated from the $d-1$-dimensional Hausdorff measures of the foliae: Theorem 3.2.11 from \cite{Federer69}, the ``coarea formula'' says that if $f:\IR^m\to\IR^n$ is Lipschitz continuous, $A\subset \IR^m$ is Lebesgue measurable and $m>n$ , then
\begin{equation}
\int_A J_n f \diff Leb^m = \int_{\IR^n} \mathcal{H}^{m-n}(A\cap f^{-1}\{y\} )\diff Leb^n(y).
\end{equation}
We apply this with $m=d$ and $n=1$ to the above function $f(x)=d(x,H)$. Since $A\subset \IR^m$ and $T_\Delta A\subset \IR^m$ are both disjoint from $\bar{H}$, $J_n f= 1$ almost everywhere on them, and the theorem gives that
\begin{equation}\label{eq:Leb-decomposition-A}
Leb(A)=\int_0^\infty \mathcal{H}^{d-1} (\{x\in A\,|\,d(x,H)=t\}) \diff t,
\end{equation}
\begin{equation}\label{eq:Leb-decomposition-TA}
Leb(T_\Delta A)=\int_0^\infty \mathcal{H}^{d-1} (\{x\in T_\Delta A\,|\,d(x,H)=t\}) \diff t.
\end{equation}
But
\begin{equation}
\{y\in T_\Delta A\,|\, d(y,H)=t\}=T_\Delta(\{x\in A\,|\, d(x,H)=t+\Delta\}),
\end{equation}
so Proposition~\ref{prop:T_Delta_weak_contraction} implies
\begin{equation}
\mathcal{H}^{d-1} (\{y\in T_\Delta A\,|\, d(y,H)=t\})
\ge \left(\frac{R-\Delta}{R}\right)^{d-1}\mathcal{H}^{d-1}(\{x\in A\,|\, d(x,H)=t+\Delta\}).
\end{equation}
Writing this back to (\ref{eq:Leb-decomposition-TA}) and (\ref{eq:Leb-decomposition-A}), we get
\begin{eqnarray}
 Leb(T_\Delta A)&\ge& \left(\frac{R-\Delta}{R}\right)^{d-1} \int_0^\infty \mathcal{H}^{d-1} (\{x\in A\,|\, d(x,H)=t+\Delta\})\diff t = \nn\\
 &=& \left(\frac{R-\Delta}{R}\right)^{d-1}Leb(A).
 \end{eqnarray}
\end{proof}

\begin{remark}\label{rem:TA_measurable}
\emph{[Measurability of $T_\Delta A$]}.
On the full measure set of $x$ where $\pi(x)$ is well defined, $T_\Delta$ is also well defined. Moreover, by Proposition~\ref{prop:T_Delta_weak_contraction} the inverse of $T_\Delta$ is Lipschitz continuous and thus Lebesgue measurable. So if $A\subset\IR^d$ is Lebesgue measurable, then so is $T_\Delta A$.
\end{remark}

\section{Discussion}\label{sec:discussion}

\subsection{Closed spheres instead of open ones}
\label{sec:what-if-closed}
We now prove Remark~\ref{rem:what-if-B-closed}. We start with a proposition.

\begin{proposition}\label{prop:continuity-in-r}
 If $D\subset\IR^d$ is Lebesgue measurable with $Leb(Conv(D))<\infty$ and $f:D\to\IR$ is bounded, then the functions
 \begin{equation}
  r\mapsto G_1(r):=\int_D \sup_{y\in B_r(x)} f(y) \diff\mu(x),
 \end{equation}
 \begin{equation}
  r\mapsto G_2(r):= \int_D \inf_{y\in B_r(x)} f(y) \diff\mu(x),
 \end{equation}
 \begin{equation}
  r\mapsto I(r):=\int_D osc_r f \diff\mu
 \end{equation}
 are continuous at every $r>0$.
\end{proposition}

\begin{proof}
 Continuity of $G_1$ was already stated as Remark~\ref{rem:continuity-in-r}. Continuity of $G_2$ is a trivial consequence substituting $f\to(-f)$. Eventually, $I=G_1-G_2$.
\end{proof}

\begin{proof}[Proof of Remark~\ref{rem:what-if-B-closed}]
Item~\ref{it:closed-mainthm} was shown in Remark~\ref{rem:closed-osc_delta_g1_est-h1-h2}. To see the rest, assume first that $Leb(Conv(D))<\infty$ and $f$ is bounded. Fix some $r>0$ and let
\begin{equation}
 (osc_{r+0}f)(x)=\lim_{R\searrow r} (osc_R f)(x)
\end{equation}
for every $x\in D$, which exists, since $(osc_R f)(x)$ is monotone increasing in $R$.
Then Proposition~\ref{prop:continuity-in-r} and the monotone convergence theorem imply that 
\begin{equation}
 \int_D osc_r f \diff Leb = \int_D osc_{r+0}f \diff Leb.
\end{equation}
On the other hand,
 $osc_r f  \le osc_{r+0}f$,
so they must be equal almost everywhere. Eventually,
 $osc_r f \le \overline{osc}_r f \le osc_{r+0}f$
implies that $osc_r f = \overline{osc}_r f$ almost everywhere.

In the general case, when $Leb(Conv(D))=\infty$ and/or $f$ is unbounded, we can truncate: for $N\in\IN$ let $D_N=D\cap B_N(0)$ and let $f_N:D_N\to \IR$ be defined by
$f_N(x):=\min\{\max\{-N,f(x)\},N\}$. Then every $F_N$ is bounded with a bounded domain, $(osc_r f)(x)=\lim_{N\to\infty}(osc_r f_N)(x)$ and $(\overline{osc}_r f)(x)=\lim_{N\to\infty}(\overline{osc}_r f_N)(x)$ for every $x\in D$, so item~\ref{it:closed-osc-thesame} is proven. This in turn implies items \ref{it:closed-osc-measurable} and \ref{it:closed-osc-thesame}.
\end{proof}

\subsection{Optimality of Theorem~\ref{THM:OSC_F-GENHOLDER}}

In the statement of Theorem~\ref{THM:OSC_F-GENHOLDER}, the $r$-dependence of the form $\frac{1}{r^\alpha}$ is optimal: If $c=1$, $d=1$, $D=[0,L]$ and $r$ is small, consider $f$ to be the indicator of $D\cap 4r\IZ$. An easy calculation gives that
\begin{equation}
 \int_D osc_\delta (osc_r f) \diff Leb \begin{cases}
                                        = L & \text{ if $\delta\ge r$} \\
                                        \approx \frac{\delta}{r} L & \text{ if $\delta < r$}
                                       \end{cases}
\end{equation}
This gives
\begin{equation}
 |osc_r f|_{\alpha;gH}=\sup_{\delta>0}\frac{1}{\delta^{\alpha}}\int_D (osc_\delta (osc_r f))\diff Leb \approx \frac{\mu(D)}{r^\alpha}
\end{equation}
(with the supremum actually taken near $\delta=r\neq 0$ whenever $\alpha<1$). Similar examples can be constructed in higher dimensions.

The statement is also optimal in the sense that $\mu(Conv(D))$ on the right hand side can not be replaced by $\mu(D)$ or $\mu(\bar{D})$: if $d=1$, consider some $N>1$ and $D=[-N-1,-N+1]\cup\{0\}\cup[N-1,N+1]$, let $f$ be the indicator of $\{0\}$ and let $r=N$. Then an easy calculation gives $|osc_r f|_{\alpha;gH}=\mu(D)$ for every $0<\alpha\le 1$, irrespective of how big $r=N$ is. If $d=2$, a similar example can be constructed with $D$ open and simply connected: let
$D=((-N-1,-N+1)\times (-1,1))\cup((N-1,N+1)\times (-1,1))\cup((-N-1,N+1)\times(-1/N^{10},1/N^{10}))$, let $f$ be the indicator of $\{0\}$ and let $r=N$.
Again, $|osc_r f|_{\alpha;gH}\approx\mu(D)$ when $r=N$ is big. 

On the other hand, there seems to be no reason why the coefficient $2(2d+1)^\alpha$, which multiplies $(\sup_D f - \inf_D f) \frac{\mu(Conv(D))}{r^\alpha}$ in the statement, would be optimal. In fact, the estimate (\ref{eq:osc_r_f_estimate-with-g1_g2}) is very rough, but (\ref{eq:g1bound}) and (\ref{eq:g2bound}) are likely to be non-optimal as well.

\subsection{Possible generalization}

Instead of $D\subset\IR^d$, consider $D\subset\mathcal{M}$, where $\mathcal{M}$ is some Riemannian manifold. Let $f:D\to\IR$. Then $osc_r f$ and $|f|_{\alpha;gH}$ still make sense, by just using the Riemannian metric to measure distance and the canonical measure for integration. In this case, the conjecture of the author is that $osc_r f$ is still automatically generalized Hölder continuous for any $r>0$, whenever $D$ and $f$ are bounded, with a bound on $|osc_r f|_{\alpha;gH}$ similar to the one given in Theorem~\ref{THM:OSC_F-GENHOLDER}. However, a direct adaptation of the present proof would be nontrivial: the logically first step (which is Lemma~\ref{lem:T_t_weak_contraction}) already breaks down. The present proof also relies on Theorem 3.2.11 from \cite{Federer69}, which is only stated and proven for Euclidean spaces. The detailed discussion is beyond the scope of this paper.

\end{document}